\documentclass[12pt,a4paper,reqno]{amsart}

\usepackage{amsrefs}

\usepackage{calc,verbatim,enumitem,tikz,url,hyperref,mathrsfs,cite}

\usepackage{thm-restate}
\usepackage{enumitem}
\usepackage{float}
\usepackage{xcolor}
\usepackage{marginnote}
\usepackage{mathrsfs}
\usepackage{extarrows}
\usepackage{calligra}
\usepackage{tabularx}
\usepackage{booktabs}
\usepackage{pgfplots}

\pgfplotsset{
	standard/.style={
		axis x line=middle,
		axis y line=middle,
		every axis x label/.style={at={(current axis.right of origin)},anchor=north west},
		every axis y label/.style={at={(current axis.above origin)},anchor=north east}
	}
}

\pgfplotsset{width=13cm,compat=1.13}

\usepgfplotslibrary{fillbetween}

\usepackage{amsmath, amssymb, latexsym, amsthm, amsbsy}

\usepackage{graphics}

\usepackage{graphicx,relsize}

\usepackage[color, all]{xy}

\usepackage[mathscr]{eucal}

\usepackage{bbm}

\usepackage[a4paper,top=3cm,bottom=3cm,inner=3cm,outer=3cm,includehead,includefoot]{geometry}
\newcommand{\subgp}[1]{\langle{{\hash}1}\rangle}

\def\y{\textbf{y}}

\def\Bin{\mathbf{Bin}}

\newcommand{\eq}[1]{\begin{equation}\label{eq:#1}}
	\newcommand{\eqe}{\end{equation}}

\numberwithin{equation}{section}

\usepackage{thmtools, thm-restate}

\newtheorem{theorem}{Theorem}[section]
\newtheorem{lem}[theorem]{Lemma}
\newtheorem{conjecture}[theorem]{Conjecture}

\newtheorem{prop}[theorem]{Proposition}
\theoremstyle{definition}\newtheorem{definition}[theorem]{Definition}
\theoremstyle{definition}

\def\le{\leqslant}

\renewcommand{\leq}{\leqslant}
\renewcommand{\geq}{\geqslant}

\usepackage{soul}

\definecolor{sgreen}{rgb}{0.3, 0.9, 0.3} 

\definecolor{lblue}{rgb}{0.6, 0.6, 1} 
\definecolor{lgr}{rgb}{0.8, 0.8, 0.8} 
\definecolor{purp}{rgb}{0.9, 0, 0.9} 

\definecolor{mgr}{rgb}{0.7, 0.7, 0.7} 
\definecolor{dmgr}{rgb}{0.6, 0.6, 0.6} 

\pgfplotsset{compat=1.13}% <- if you have an older installation, try 1.15 or 1.14
\usepgfplotslibrary{groupplots,fillbetween}

\begin{document}

\title[On Meyniel's conjecture in Random Hypergraphs]{On Meyniel's conjecture in Random Hypergraphs}
	
	\author{Gabriel Dias do Couto}

\address{Departamento de Ensino, IFAM -- Campus São Gabriel da Cachoeira, BR-307, 69750-000 Rio de Janeiro, Brazil}\email{gabriel.couto@ifam.edu.br,}

%\thanks{Gabriel is supported by a CAPES PhD scholarship and also by FAPERJ.\\ This study was financed in part by the Coordenação de Aperfeiçoamento de Pessoal de Nível Superior – Brasil (CAPES) – Finance Code 001}

	\begin{abstract}
		The game of \emph{Cops and Robbers} is a two player pursuit game on graphs where a team of cops attempts to catch a robber. The cop number $c(G)$ of a graph $G$ is the minimum number of cops needed to guarantee a winning strategy in $G$. A famous conjecture of Meyniel says that if $G$ is connected, then $c(G)=O(\sqrt{n})$. Erde, Kang, Lehner, Mohar and Schmid considered its generalization to $k$-uniform hypergraphs and conjectured that the cop number of such hypergraphs is $O(\sqrt{n/k})$. This may be understood as a hypergraph version of Meyniel's conjecture. In this paper we prove this conjecture for a class of \textit{expanding} hypergraphs and show that with high probability the conjecture holds for random hypergraphs $H^k(n,p))$ provided $k\geq \log^3n$ and the typical degree, $p\binom{n-1}{k-1}$, is $\omega(\log^3n)$.
	\end{abstract}

	\maketitle

\section{Introduction}

The ``Cops and Robbers game'' was first introduced by Quilliot \cite{Q} and Nowa-kowski and Winkler \cite{NW}. The game takes place in a graph $G$ and is played as follows: first the player who will control the cops chooses their initial positions. Then the player who controls the robber chooses her initial position with full information. Now the Cop Player and the Robber Player move their pieces in alternate turns, the Cop Player moving first. The Cop Player may move any number of pieces as he wants (or choose not to). The Robber Player then makes the choice to move her piece or stay still. All moves must be to a neighbouring vertex. Each player has full information about the graph and the pieces, at all times. The Cop Player wins if at some finite time one of his pieces occupies the same vertex as the robber piece. On the other hand, if the Robber Player can escape the cops indefinitely, she wins.

Shortly after the work of \cite{NW}, Aigner and Fromme \cite{AF} defined the concept of \emph{cop number}. The cop number of a graph $G$, $c(G)$, is the minimum number of cops needed for the Cop Player to have a winning strategy in $G$. The cop number is well defined since $n$ cops, one for each vertex, always catch the robber. They also proved, among other results, that if $G$ is planar and connected, then $c(G)\leq 3$.

It can easily be seen that $c(G)=\sum_{G_i}c(G_i)$, where the $G_i$'s are the connected components of $G$. For this reason it suffices to study the case where $G$ is connected. The most important conjecture on the game is Meyniel's conjecture (communicated by Frankl \cite{F}).
\begin{conjecture}
    If $G$ is a connected graph with $n$ vertices, then $c(G)=O(\sqrt{n}$).
\end{conjecture}

The conjecture is tight in the sense that there are connected graphs such that $\sqrt{n}/2$ cops are not enough to catch the robber (see Bonato and Burgess \cite{BB}). We are still very far from the conjecture, since even its weak form, which asserts that $O(n^{1-\varepsilon})$ cops are sufficient (for some $\varepsilon>0$), is still unknown.  The general bounds of $O(n\log\log n/\log{n})$ and $O(n/\log{n})$ were proved in \cite{F} and \cite{C} respectively.  The current best general upper bound is $n2^{-(1+o(1))\sqrt{\log_2n}}$, which was proved independently by Lu and Peng \cite{LuP}, Scott and Sudakov \cite{SS} and Frieze, Krivelevich and Loh \cite{FKL}. As random structures are often used to try to understand a general scope of a problem, it is natural to ask about the random graph case.

Considering the Erd\H{o}s-R\'eyni $G(n,p)$ random graph model, Bonato, Hahn and Wang \cite{BHW} first showed that if $p$ is constant then $c(G(n,p))$ is logarithmic in $n$. For $p=n^{\alpha-1}$, $\alpha>0$, \L uczak and Pra\l at \cite{LP} found evidence to support Meyniel's conjecture, showing a Zig-Zag behaviour of $c(G(n,p))$ depending on $\alpha$.

\begin{theorem}\label{thm:zigzagG}\cite{LP}
    Let $\alpha>0$ and $p=n^{\alpha-1+o(1)}$, then with high probability
    $$c(G(n,p)) \,=\, O(\sqrt{n}\cdot \log n).$$
\end{theorem}
Furthermore, their result shows that the function $f:(0,1)\to[0,1]$
$$f(\alpha) \,=\, \frac{\log(\bar{c}(G(n,n^{\alpha-1})))}{\log n}$$
behaves as in Figure \ref{fig:ZigZagG}. Here, $\bar{c}(G(n,n^{\alpha-1}))$ is the median of the cop number for $G(n,n^{\alpha-1})$.
\begin{figure}[ht]
    \centering
    \includegraphics[width=0.7\linewidth]{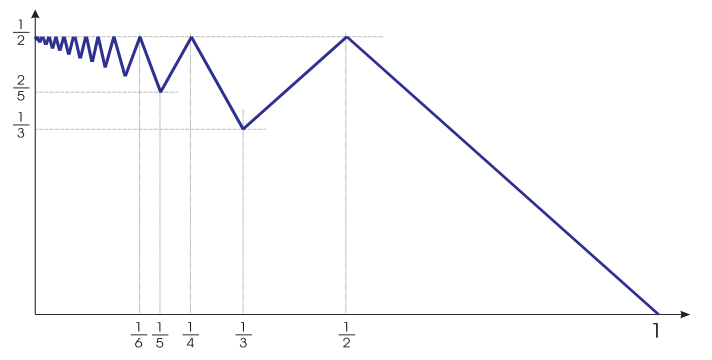}
    \caption{The ZigZag behavior of $f(\alpha)$. Reprinted from \cite{LP}.}
    \label{fig:ZigZagG}
\end{figure}

Their strategy consists of using a team of cops to enclose the robber in some neighbourhood of size $O(\sqrt{n})$.  Bollobás, Kun and Leader \cite{BKL} proved the same bound, although for sparser regimes of $p$. Refining the idea of \cite{LP}, Pra\l at and Wormald \cite{PW} made a multi-team pursuit strategy depending on the density of the random graph and proved Meyniel's conjecture for random graphs for $p$ greater than the connectivity threshold.

\begin{theorem}\label{thm:zigzagH}\cite{PW}
    Let $\varepsilon>0$ and $p\geq(\frac{1}{2}+\varepsilon)\log n/n$, then with high probability
    $$c(G(n,p)) \,=\, O(\sqrt{n}).$$
\end{theorem}

For a hypergraph $H$, one may similarly define the Cops and Robbers game.  The only difference is that one may now move to any vertex which shares an edge with the current vertex.  In this way, the game is equivalent to the graph based game in which each edge of $H$ is replaced by a clique.  The cop number of $H$, $c(H)$, is similarly defined as the minimum number of cops needed to ensure that the Cop Player has a winning strategy in $H$. This problem was first considered by Gottlob, Leone and Scarcello \cite{GNF} and Adler \cite{A}. In a recent article, Erde, Kang, Lehner, Mohar and Schmid \cite{HyperLog} introduce the following conjecture, which generalizes Meyniel's conjecture. 

%If you replace a $k$-hyperedge by a $k$-clique, you find a class of graphs that are equivalent to the $k$-uniform hypergraphs for this game. For the reverse, given a graph $G$, consider a $k$-blow-up $G(k)$ of it. An edge $uv\in G$ generates a complete bipartite subgraph $\{u_1,\dot\,u_k,v_1,\dots,v_k\}$. In this setting we get a $2k$-uniform hypergraph if we put a $2k$-hyperedge in place of this bipartite subgraph. They use these ideas to conjecture a generalization of Meyniel's conjecture.

\begin{conjecture}\cite{HyperLog}\label{conj:conj}
    If $H$ is a connected $k$-uniform hypergraph with $n$ vertices, then $c(H)=O(\sqrt{n/k})$.
\end{conjecture}

The main theorem of \cite{HyperLog} is in the setting of random $k$-uniform hypergraphs $H^k(n,p)$. There they used, as in \cite{LP}, a team of cops to block the escape of the robber, though with two strategies, known as the edge-surrounding strategy and vertex-surrounding strategy.  The strategy used depends on the parameters $k$ and $p$.  This mix of strategies gave them a dual-zigzag result represented in Figure \ref{fig:ZigZagH}.  As a consequence of their zigzag result they obtain the following bound.

\begin{theorem}\cite{HyperLog}
    If $k=\omega(\log n)$ and $\frac{n}{k}\geq p\binom{n-1}{k-1}=\omega(\log^3 n)$, then with high probability
    $$c(H^k(n,p)) \,=\, O\left(\sqrt{n/k}\cdot \log n\right).$$
\end{theorem}
\begin{figure}[ht]
    \centering
    \includegraphics[width=0.7\linewidth]{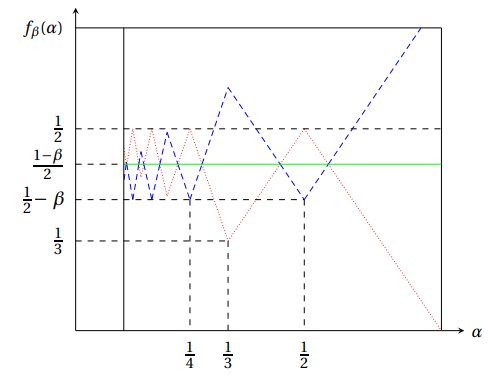}
    \caption{In this setting $p=n^{\alpha-1}$, $k=n^{\beta}$, and  define $f_\beta(\alpha)=\frac{\log(c(H^k(n,p))}{\log n}$ ($\beta=2/19$ in the graph). The blue dashed line represents the edge-surrounding strategy result and the pink one the vertex-surrounding strategy. Note that, at the values of $\alpha$ for which the strategies coincide, the cop number is $n^{\frac{1-\beta}{2}(1+o(1))}\approx \sqrt{n/k}\cdot n^{o(1)}$. Reprinted from \cite{HyperLog}.}
    \label{fig:ZigZagH}
\end{figure}

In this work, we combine and adapt ideas from \cite{HyperLog} and \cite{PW} to remove the log-factor and therefore prove Conjecture~\ref{conj:conj} for these random hypergraphs. Note we require a slightly stronger lower bound on $k$.

\begin{restatable}{theorem}{main}\label{thm:main}	 
	If $k\geq \log^3 n$ and $n/k\geq p\binom{n-1}{k-1}=\omega(\log^3 n)$, then with high probability
    $$c(H^k(n,p)) \,=\, O\left(\sqrt{n/k}\right).$$
\end{restatable}

%We observe that the stronger upper bound won't hurt the result. Take a any set $A$ of $\sqrt{n/k}$ vertices. Then if $p\binom{n-1}{k-1}\geq \frac{2n}{k\log n}$, we have that $A$ dominates the random hypergraph with high probability.

\textbf{Layout of the paper.} In the next section we give an overview of our approach. This includes existing strategies for the Cop Player and a discussion of how they may be adapted and combined. In Section \ref{sec:key} we prove some key lemmas that we use in Section \ref{sec:det} to show that every $k$-uniform hypergraph in a class of hypergraphs with some expanding properties has cop number $O(\sqrt{n/k})$. In Section \ref{sec:expanding} we will show that with high probability a random hypergraph $H^k(n,p)$ falls into this class, which concludes the proof of our main theorem as a corollary. Finally, in Section \ref{sec:remarks} we discuss what is proved and what remains open for all ranges of $k$ and $p$ in this setting.
	
\section{An overview of the proof}\label{sec:overview}

Firstly it will be very important to define a metric over $V(H)\cup E(H)$.  For the graph case the usual metric is realised by adding an extra vertex in the middle of each edge and computing the graph distance from there. For the hypergraph case we do not have this interpretation, since the edges have more than two vertices, but the reader may keep this idea in mind for what follows.

Two vertices $u$ and $v$ which belong to a common edge have distance $d(u,v)=1$. To extend this metric naturally to $V(H)\cup E(H)$, if $e$ is an edge that contains $v$, then we define $d(v,e)=1/2$. Observe that, with this definition, for $v\in V(H)$ and distinct $e,g\in E(H)$ we have $d(v,e)=\min\{d(u,v):u\in e\}+1/2$ and $d(e,g):=\min\{d(u,v):u\in e,v\in g\}+1$.

Now we are able to properly define neighbourhoods and spheres on the hypergraph. For a vertex $v\in V(H)$ and $r\in\mathbb{N}$, let $N_V(v,r):=\{u\in V(H):d(u,v)\leq r\}$ and $N_E(v,r-1/2):=\{e\in E(H):d(e,v)\leq r-1/2\}$ be, respectively, the $r$th vertex and $r$th edge neighbourhood of $v$. Also, let $S_V(v,r):=\{u\in V(H):d(u,v)= r\}$ and $S_E(v,r-1/2):=\{e\in E(H):d(e,v)= r-1/2\}$ be the related $r$-spheres. One may generalize these definitions to edges as well remembering that for a given edge $e$, the distance from $e$ to its vertices is $1/2$ and computing distances from there.

In \cite{HyperLog,LP,PW}, the authors first prove that for a given class of expanding (hyper)graphs a certain number of cops is sufficient, and then show that the random (hyper)graphs are in that class with high probability. Our approach is based on adapting and combining strategies from \cite{HyperLog} and \cite{PW}. We begin by discussing the two most important strategies --- the \emph{vertex-surrounding} and \emph{edge-surrounding} strategies. After that we show how we may use both at the same time with more teams of cops.

The vertex-surrounding strategy is based on the following idea: If the robber starts at $v\in V(H)$ and for a given $u\in S_V(v,r)$, there is a cop piece in $N_V(u,r+1)$, then as the Cop Player moves first, he has the option to send this cop to a neighbour of $u$ before the robber can reach $u$. Taking this idea further, if we can define an injective function $f:S_V(v,r)\to C$, from $S_V(v,r)$ to the set of initial positions of the cop pieces, such that $d(u,f(u))\leq r+1$ for every $u\in S_V(v,r)$, the cops can trap the robber in $N_V(v,r-1)$. It is then not difficult to catch the robber. This strategy is illustrated in Figure \ref{fig:vertex}.

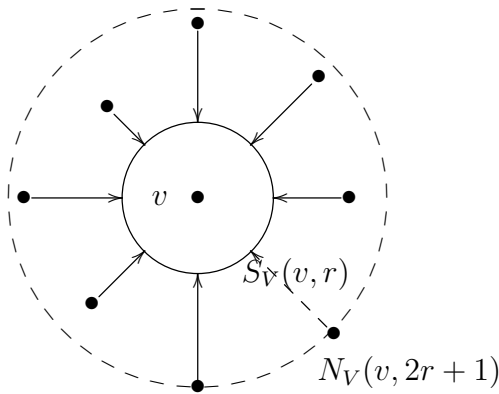
\begin{figure}[H]
		$$	
		\begin{xy}
			(0,0)*{\bullet};
			%(0,20)*{\bullet};
			(-5,0)*{v};
			%(-5,20)*{w};
			(0,0)*\xycircle(10,10){};
			(0,-10)*{}; (0,-25)*{\bullet} **\dir{-} ?<* \dir{<};
			(10,0)*{}; (20,0)*{\bullet} **\dir{-} ?<* \dir{<};
            (-10,0)*{}; (-23,0)*{\bullet} **\dir{-} ?<* \dir{<};
            (0,10)*{}; (0,23)*{\bullet} **\dir{-} ?<* \dir{<};
			(0,0)*\xycircle(25,25){--};
			(28,-23)*{N_V(v,2r+1)};
			(13,-10)*{S_V(v,r)};
            (7,7)*{}; (16,16)*{\bullet} **\dir{-} ?<* \dir{<};
            (-7,-7)*{}; (-14,-14)*{\bullet} **\dir{-} ?<* \dir{<};
            (7,-7)*{}; (18,-18)*{\bullet} **\dir{--} ?<* \dir{<};
            (-7,7)*{}; (-12,12)*{\bullet} **\dir{-} ?<* \dir{<};
			%(30,-5)*{}; (0,20)*{} **\dir{-};
			%(30,25)*{}; (0,20)*{} **\dir{-};
		\end{xy}
		$$
		
		\caption{In this figure we illustrate the vertex-surrounding strategy. As the cops move first, they have an extra step to reach $S_V(v,r)$.}
		\label{fig:vertex}
	\end{figure}

Another strategy is to surround the robber using its edge-spheres. The edge-surrounding strategy involves sending a cop to each $e\in S_E(v,r-1/2)$ in their first $r$ steps. Note that a cop starting in $N_V(e,r+1/2)$ may reach $e$ in $r$ steps. The difference from the vertex-surrounding strategy is not only that it involves edges, but also that the cop piece must be within distance $2r$ from the robber, instead of $2r+1$. We illustrate this obstruction in Figure \ref{fig:draw}. Hence, it is more difficult to guarantee the family of available cops. On the other hand, as one starts from an edge, one gains a factor of $k$ on the size of the neighbourhood as the edge has $k$ vertices in it. This partially compensates for the loss from the distance.%  That is, if $e\in S_E(v,r-1/2)$ and there is a cop piece in $N_V(e,r+1/2)$, then the cop can reach this edge in round $r$ before the robber can leave it. Hence, if we can define an injective function $f:S_E(v,r-1/2)\to C$, such that $d(e,f(e))\leq r+1/2$, the cop player can win. Note that in the edge-surrounding strategy the cops need to be within distance at most $2r$ instead of $2r+1$ from $v$.  This makes it more difficult to guarantee the family of available cops. On the other hand, as one starts from an edge, one gains a factor of $k$ on the size of the neighborhood as the edge has $k$ vertices in it. This counterbalances the losses from the distance.

\begin{figure}[H]
	\centering
	\begin{tikzpicture}[scale=0.12]
		
		%------------------------------------------------
		% Leftmost vertex (robber)
		%------------------------------------------------
		\node[circle,fill,inner sep=2pt,label=below:$v$] (R) at (0,0) {};
		\node at (0,10) {Distance:};
		
		%------------------------------------------------
		% First circle and vertices
		%------------------------------------------------
		\draw (8,0) ellipse (12 and 4);
		
		\node[circle,fill,inner sep=2pt] (F) at (16,0) {};
		\node at (16,10) {1};
		
		\node[circle,fill,inner sep=2pt] at (11.5,  2.5) {};
		\node[circle,fill,inner sep=2pt] at (11.5, -2.5) {};
		
		%------------------------------------------------
		% Second circle
		%------------------------------------------------
		\draw (24,0) ellipse (12 and 4);
		
		\node[circle,fill,inner sep=2pt] at (32,0) {};
		\node at (32,10) {2};
		
		%------------------------------------------------
		% Third circle and cop
		%------------------------------------------------
		\draw (40,0) ellipse (12 and 4);
		
		\node[circle,fill,inner sep=2pt] (C) at (48,0) {};
		\node at (48,10) {3};
		
		%------------------------------------------------
		% Auxiliary point FR (invisible)
		%------------------------------------------------
		\node (FR) at (14,8) {};
		
		%------------------------------------------------
		% Curved arrows
		%------------------------------------------------
		\draw[->,blue,thick]
		(C) to[out=120,in=-50] (F);
		
		\draw[->,red,thick]
		(R) .. controls (12,2) .. (FR);
		
	\end{tikzpicture}
	
	\caption{In this figure we illustrate the necessity of a smaller distance from the cops to the robber in the edge-surrounding strategy. Here $r=1$ and the cop piece is at distance $2r+1=3$ from the robber. The arrows represent two moves of each player, where the robber starts at $v$ and the cop at some vertex within distance $3$ from it. When the cop piece reaches the edge in $S_E(v,1/2)$, the robber is already positioned to escape using one of the vertices not occupied by the cop, causing the strategy to fail.}
	\label{fig:draw}
\end{figure}
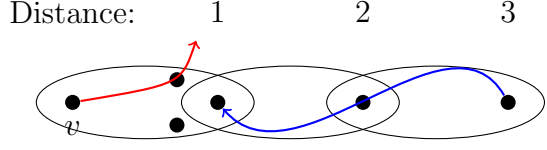

We now define a class of hypergraphs, with expansion properties which will perform a key role for this paper.
\begin{definition}\label{def:exp}
	For $\alpha\in(0,1)$ and $d>0$, a $k$-uniform hypergraph $H^k=(V,E)$ with $n$ vertices is $(d,\alpha)$-\emph{expanding} if it has the following properties:
	\begin{enumerate}
		\item\label{proper:1} For all $v\in V$ and $r\in\mathbb{N}$ satisfying $d^r\leq \sqrt{nk}$,%\cdot \log n$,
		$$\alpha\frac{d^r}{k}\,\leq\, \left|N_E\left(v,r-\frac{1}{2}\right)\right| \,\leq\, \alpha^{-1}\frac{d^r}{k}$$
		\item\label{proper:2} For all $A\subset V$ and $r\in\mathbb{N}$ $$\alpha\min\{|A|d^r,n\} \,\leq\, |N_V(A,r)| \,\leq\, \alpha^{-1}|A|d^r$$
		\item\label{proper:3} For all $B\subset E$ and $r\in\mathbb{N}$
		$$|N_V(B,r+1/2)|\, \geq\, \alpha\min\{|B|kd^r,n\}$$
		\item\label{proper:4} Let $r\in\mathbb{N}$ such that $\sqrt{nk}<d^{r+1}\leq\sqrt{nk}\log n$. For all $v\in V$ there exists a family 
		$$\big\{W(u)\subset N_V(u,r+1)\,:\,u\in S_V(v,r)\big\}$$ of pairwise disjoint subsets such that, for each $u\in S_V(v,r)$ 
		$$|W(u)| \,\geq\, \alpha d^{r+1}$$
		\item\label{proper:5} Let $r\in\mathbb{N}$ such that $\sqrt{\frac{n}{k}}<d^{r}\leq\sqrt{\frac{n}{k}}\log n$. For all $v\in V$ there exists a family 
		$$\left\{W(e)\subset N_V\left(e,r+\frac{1}{2}\right)\,:\,e\in S_E\left(v,r-\frac{1}{2}\right)\right\}$$ of pairwise disjoint subsets such that, for each $e\in S_E(v,r-1/2)$ 
		$$|W(e)| \,\geq\, \alpha kd^{r}$$
	\end{enumerate}
\end{definition}

We now briefly explain the intuition behind the definition. Let $\tilde{d}:=\frac{1}{n}\sum_{v\in V}|S_V(v,1)|$ be the average vertex degree of $v$. For the $k$-uniform random hypergraph $H^k(n,p)$, note that $\hat{d}:=kp\binom{n-1}{k-1}$ is close to $\tilde{d}$. In fact, \cite{HyperLog} showed that
 $$2^{-37}\tilde{d}\leq \hat{d}\leq 2^{37}\tilde{d}.$$
 Hence, since for the random setting we have good concentration for the vertex and edge expansions (see Theorem~\ref{thm:hyperlog-exp}), the definition above gives us that: if the random hypergraph $H^k(n,p)$ is $(\hat{d},\alpha)$-expanding, then $|N_V(v,r)|=\Theta(\hat{d}^r)$ and $|N_E(v,r-\frac{1}{2})|=\Theta(\hat{d}^r/k)$.

For the cases where we only use one team of cops, the following lemmas are key. Moreover, we will also use them multiple times when we need more than two teams of cops. They essentially say that if an hypergraph is \textit{well-behaved} in expansion terms, then any \textit{small} set of vertices (edges) may be covered with a \textit{well-distributed} team of cops within few steps.

\begin{restatable}{lem}{vertices}\label{lem:vertices}
	For all $\alpha\in(0,1)$ and $\gamma\geq 1$, there is a $C>0$ such that for every $(d,\alpha)$-expanding $k$-uniform hypergraph $H^k=(V,E)$ with $|V|=n$, the following holds.
	
	Let $r\in\mathbb{N}$ be such that $d^r\geq \sqrt{nk}\log n$. For a given $X\subset V$ with $|X|\leq\gamma\sqrt{\frac{n}{k}}$, if $Y\subset V$ is a $\frac{C}{\sqrt{nk}}$-random subset of $V$, then with probability at least $1-n^{-3}$ there is an injection $f:X\to Y$ such that $d(x,f(x))\leq r$, $\forall x\in X$.
\end{restatable}
\begin{restatable}{lem}{edges}\label{lem:edges}
	For all $\alpha\in(0,1)$ and $\gamma\geq1$, there is a $C>0$ such that for every $(d,\alpha)$-expanding $k$-uniform hypergraph $H^k=(V,E)$ with $|V|=n$, the following holds.
	
	Let $r\in\mathbb{N}$ be such that $d^r\geq \sqrt{n/k}\log n$. For a given $X\subset E$ with $|X|\leq\gamma\sqrt{\frac{n}{k}}$, if $Y\subset V$ is a $\frac{C}{\sqrt{nk}}$-random subset of $V$, then with probability at least $1-n^{-3}$ there is an injection $f:X\to Y$ such that $d(e,f(e))\leq r+\frac{1}{2}$, $\forall e\in X$.
\end{restatable}

For certain values of $d$, Lemma \ref{lem:vertices} will give us exactly what we need to apply the vertex-surrounding strategy. In particular, if $d^r\leq\sqrt{n/k}$ and $d^{r+1}\geq\sqrt{nk}\log n$, then we argue as follows: 

For each starting position $v\in V(H^k)$ of the Robber Player, we will apply Lemma \ref{lem:vertices} with $X=X(v)=S_V(v,r)$. As $H^k$ is $(d,\alpha)$-expanding, $|X(v)|\leq \alpha^{-1}\sqrt{n/k}$, for all $v$. Hence, since $d^{r+1}\geq \sqrt{nk}\log n$ we have that there exists a realization of $Y$ such that there is an injective function $f_v:X(v)\to Y$ with $d(x,f_v(x))\leq r+1$, for all $v$. The Cop Player then will choose to place a cop in each vertex of $Y$ and will dominate the chosen $X(v)$ in $r+1$ turns, following $f_v$, surrounding the robber. In this work, a set $A$ dominates a set $B$ if $B\subset A\cup N_V(A,1)$.

Analogously, if $d^r\geq\sqrt{n/k}\log n$ and $d^{r+1}\geq\sqrt{nk}\log n$ we may use the edge-surrounding strategy.  In light of that, we break our strategy into four cases depending on the values of $d^r$ and $d^{r+1}$, as in Figure \ref{fig:graphic},
\begin{itemize}
    \item \textbf{Case I:} $d^r\leq\sqrt{n/k}$ and $d^{r+1}\geq\sqrt{nk}\log n$;
    \item \textbf{Case II:} $d^r\leq\sqrt{n/k}$ and $d^{r+1}\leq\sqrt{nk}\log n$;
    \item \textbf{Case III:} $d^r\geq\sqrt{n/k}\log n$ and $d^{r+1}\geq\sqrt{nk}\log n$;
    \item \textbf{Case IV:} $d^r\in\big[\sqrt{n/k},\sqrt{n/k}\log n\big]$ and $d^{r+1}\geq\sqrt{nk}\log n$.
\end{itemize}

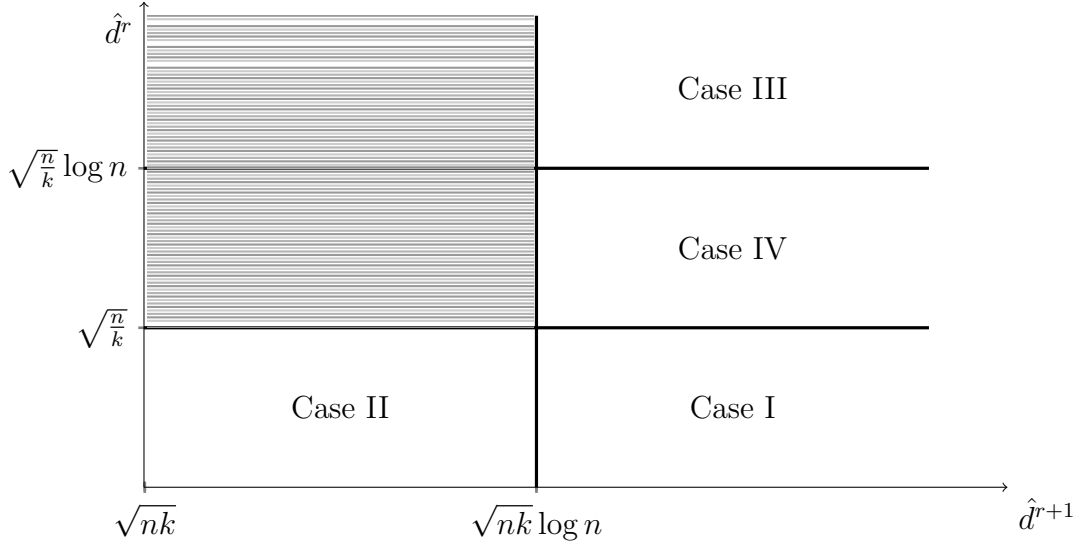
\begin{figure}[ht]
		\begin{tikzpicture}
				
			\begin{axis}[xmin=0,ymin=0,xmax=110,ymax=70,axis lines=middle, 
				standard,  axis line style={->},
				xlabel=$\hat{d}^{r+1}$,
				ylabel=$\hat{d}^r$,
				minor xtick={0.1,50},
				tick style={line width=1pt},
				xtick={0.1,50},
				xticklabels={$\vspace{1mm} \sqrt{nk}$, $\sqrt{nk}\log n$},
				%minor ytick={0,5,10,1,25.2},
				ytick={23,46},
				yticklabels={$\sqrt{\frac{n}{k}}$,$\sqrt{\frac{n}{k}}\log n$}]
				
				  \draw (25,11.5) node {Case II};
                  \draw (75,11.5) node {Case I};
                  \draw (75,57.5) node {Case III};
                  \draw (75,34.5) node {Case IV};
				
				\addplot[name path=GGS,very thick] coordinates {(0,23) (100,23)};
				\addplot[name path=NRS,very thick] coordinates {(0,46) (100,46)};
                \addplot[name path=NRS,very thick] coordinates {(50,0) (50,68)};

				\begin{scope}                                                 
					\clip (.5,23) -- (49.6,23) -- (49.6,68) -- (.5,68) -- cycle;
					\foreach \y in {23, 26,...,68} {\addplot[thick,dmgr] coordinates{(0,\y) (80,\y)};}
					\foreach \y in {24, 27,...,68} {\addplot[thick,lgr] coordinates{(0,\y) (80,\y)};}
					\foreach \y in {25, 28,...,60} {\addplot[thick,mgr] coordinates{(0,\y) (80,\y)};}
					\foreach \y in {25.5, 28.5,...,68} {\addplot[thick,lgr] coordinates{(0,\y) (80,\y)};}
					\foreach \y in {26.5, 29.5,...,68} {\addplot[thick,mgr] coordinates{(0,\y) (80,\y)};}
					\foreach \y in {24.5, 27.5,...,68} {\addplot[thick,dmgr] coordinates{(0,\y) (80,\y)};}
					
				\end{scope}

			\end{axis}
			
		\end{tikzpicture}		
		\caption{In this figure we illustrate where each case falls according to the sizes of $d^{r+1}$ and $d^r$. Note that the shaded area does not exist since $d^{r+1}/d^r> \log n$.} 
        \label{fig:graphic}
	\end{figure}

 Cases II and IV are more involved and we will need to use the Properties \ref{proper:4} and \ref{proper:5} of the definition of $(d,\alpha)$-expanding. Let us briefly discuss the extra challenges involved in Case II (Case IV is similar). Note that we cannot use Lemma \ref{lem:vertices} directly as $d^{r+1}\leq \sqrt{nk}\log n$. Hence we send a first team of cops, with the same distribution as in the lemmas, to, in some sense, densely cover $S_V(v,r)$. By doing that, we decrease by a large proportion the number of escaping routes that the robber can use to escape $N_V(v,r)$ and we may define a new set $X$, but now with edges, using these few routes. This decrease enables us to use Lemma \ref{lem:edges} to send a second team of cops to surround the edge-neighbourhoods of these routes \emph{with an extra step} and catch the robber, since $d^{r+1}\geq \sqrt{n/k}\log n$.

\section{Proof of Key Lemmas}\label{sec:key}

Throughout this paper we are going to use the following Corollary of Chernoff's Bound repeatedly.

\begin{theorem}\label{thm:chernoff}
	Let $X\sim \Bin(n,p)$ and $\mu=\mathbb{E}[X]$. Then for any $t>0$, we have
	$$\mathbb{P}\big(|X-\mu|\geq t\mu\big) \,\leq\, 2\exp\left(-\frac{t^2\mu}{3}\right).$$
\end{theorem}

We also need the Multiplicity Hall's Theorem, which provides matchings between sets, for some of our proofs.

\begin{theorem}\label{thm:Hall}
	For a finite family $\mathcal{F}$ and $l\in\mathbb{N}$, every $\mathcal{G}\subseteq\mathcal{F}$ satisfies $l\cdot\big|\mathcal{G}\big|\leq \left|\bigcup_{S\in \mathcal{G}}S\right|$ if, and only if, there are $l$ image-disjoint injective functions $f_i:\mathcal{F}\to \bigcup\mathcal{F}$ such that $f_i(S)\in S$ for each $S\in \mathcal{F}$, and for $i\in[l]$.
\end{theorem}

Now we may start to prove our key lemmas.

\vertices*

\begin{proof}
	We observe that the statement easily holds if $k> n/(\log{n})^2$ since this implies that $d^2>n$, so we may assume that $k\le n/(\log{n})^2$.

	Let $A\subset X$, with $|A|=a\leq a_0:=\max\{i:id^r<n\}$. Let $U_A=N_V(A,r)$. Since $a\leq a_0$ we have
	$$a\sqrt{nk}\log n \,\leq\, a d^r<n$$
	and as $H^k$ is $(d,\alpha)$-expanding, Property \ref{proper:2} gives us
	$$\alpha a\sqrt{nk}\log n \,\leq\, |U_A|<\alpha^{-1}n.$$
	
	Hence, the random variable $|Y\cap U_A|$ stochastically dominates the binomial random variable $\Bin\big(\alpha a\sqrt{nk}\log n, C/\sqrt{nk}\big)$, which has expected value $\mu=\alpha Ca\log n$. By Chernoff's inequality (Theorem \ref{thm:chernoff}), we have, %for $n\geq \exp(4cC/3)$,
	$$\mathbb{P}(|Y\cap U_A|\leq a) \,\leq\, \exp(-\alpha Ca\log n/6).$$%a)\leq\mathbb{P}\big(|Y\cap U_A|- \mu\leq -(\mu-a)\big)\leq\exp(-\alpha Ca\log n/6).$$
	
	By a union bound, the probability that Hall's Condition (the first part of Theorem \ref{thm:Hall}), with $\ell=1$, fails for some set $A$ with at most $a_0$ vertices is at most
	$$\sum_{a=1}^{a_0}\binom{|X|}{a}\mathbb{P}(|Y\cap U_A|\leq a) \,\leq\, \sum_{a=1}^{a_0}n^a\cdot\exp(-\alpha Ca\log n/6) \,\leq\, n^{-3}/2,$$
	for $C\geq 24/\alpha$.
	
	It remains to show the case $|A|>a_0$. So let $a_0<|A|\leq|X|\leq \gamma\sqrt{n/k}$. Then, as $H^k$ is $(d,\alpha)$-expanding, $U_A$ now is of size at least $\alpha n$, so $\mathbb{E}[|Y\cap U_A|]\geq\alpha C\sqrt{n/k}$, and hence 
	$$\mathbb{P}(|Y\cap U_A|\leq a) \,\leq\, \exp\left(-\frac{\alpha C}{6}\sqrt{n/k}\right),$$
	for $C\geq 6\gamma/\alpha$.
	A union bound shows that Hall's Condition fails with probability at most
	$$\sum_{a=a_0+1}^{|X|}\binom{|X|}{a}\mathbb{P}(|Y\cap U|\leq a) \,\leq\, 2^{\gamma\sqrt{n/k}}\cdot \exp\left(-\frac{\alpha C}{6}\sqrt{n/k}\right) \,\leq\, n^{-3}/2,$$ for $C>20\gamma/\alpha$, which concludes our proof.
\end{proof}

Now we prove the edge analogue.

\edges*

\begin{proof}
	Let $A\subset X$, with $|A|=a\leq a_0:=\max\{i:id^r<n/k\}$. Let $U=\bigcup_{u\in A}N_V(u,r-\frac{1}{2})$. Note that since $A$ is a set of edges, each neighbourhood has an extra factor of $k$, as we can see from the Property (3) of $(d,\alpha)$-expanding. Hence, by the definition of $a_0$,
	$$a\sqrt{nk}\log n\leq ka d^r<n$$
	and the rest of the proof follows as in the previous lemma.
\end{proof}

\section{Deterministic Expansion obeys Meyniel}\label{sec:det}

As we said, we will first prove a deterministic result for the cop number of $(d,\alpha)$-expanding $k$-uniform hypergraphs and then, in Section \ref{sec:expanding}, we will show that the random hypergraphs $H^k(n,p)$ are expanding with high probability

\begin{theorem}\label{thm:densedet}
    Let $k\geq \log^3 n$ and $\frac{n}{k}\geq \frac{d}{k}\geq \log ^3n$. For all $\alpha\in(0,1)$ there exists $F=F(\alpha)$ such that the following holds.
    
    Let $H^k$ be a connected $(d,\alpha)$-expanding $k$-uniform hypergraph on $n$. Then
    $$c(H^k) \,\leq\, F\sqrt{\frac{n}{k}}.$$
\end{theorem}

The proof will use the Lemmas \ref{lem:vertices} and \ref{lem:edges} repeatedly.  We always take $\gamma=\alpha^{-1}$ in our applications of these lemmas.  Let $C_{2.2}(\alpha)$ and $C_{2.3}(\alpha)$ be the constants given by the lemmas in this case.  Let $C:=C(\alpha)\geq\max\{C_{2.2}(\alpha),C_{2.3}(\alpha),10\alpha^{-2}\}$, e.g. $C=25\alpha^{-2}$.  We shall prove Theorem \ref{thm:main} with constant $F(\alpha)=7C(\alpha)$.

%In order to do that, we need to prove first Lemmas \ref{lem:vertices} and \ref{lem:edges}.

%We bring attention to the fact that we did not put any constrains on the sizes of $k$ and $d_V/k$, but they will be needed in order to properties (4) and (5) to hold. We are now ready to prove the deterministic result for the denser case.

\begin{proof}[Proof of Theorem \ref{thm:densedet}]
    Let $r=\min\{i: d^{i+1}\geq\sqrt{nk}\}$. We recall the cases stated in Section \ref{sec:overview} now and show that they exhaust all the possibilities of our theorem. 
    
    \begin{itemize}
    	\item \textbf{Case I:} $d^r\leq\sqrt{n/k}$ and $d^{r+1}\geq\sqrt{nk}\log n$;
    	\item \textbf{Case II:} $d^r\leq\sqrt{n/k}$ and $d^{r+1}\leq\sqrt{nk}\log n$;
    	\item \textbf{Case III:} $d^r\geq\sqrt{n/k}\log n$ and $d^{r+1}\geq\sqrt{nk}\log n$;
    	\item \textbf{Case IV:} $d^r\in\big[\sqrt{n/k},\sqrt{n/k}\log n\big]$ and $d^{r+1}\geq\sqrt{nk}\log n$.
    \end{itemize}
    
    We emphasize that $d/k\geq \log^3 n$ which will be used in every case. We now begin the case analysis:
	\begin{itemize}
		\item If $\sqrt{nk}\leq d^{r+1}\leq \sqrt{nk}\cdot\log n$, then 
		$$d^r\leq \sqrt{\frac{n}{k}}\cdot\log n\cdot \log^{-3}n\leq \sqrt{\frac{n}{k}}.$$
		This defines Case II.
		
		\item There may be the case where $\sqrt{nk}\cdot\log n\leq d^{r+1}$, and also that $d^r\leq\sqrt{\frac{n}{k}}$. This defines Case I.
		
		\item We observe that Cases I and II include all possibilities with $d^r\leq\sqrt{\frac{n}{k}}$ or $d^{r+1}\leq \sqrt{nk}\cdot\log n$. So it only remains to consider situations where $d^r\geq\sqrt{n/k}$, which we will choose to divide depending on $d^r$ being greater or smaller than $\sqrt{n/k}\cdot\log n$. This defines Cases III and IV.
		
	\end{itemize}
    
    We now give detailed arguments for the four cases.  In all cases, we may take the cops to occupy the set given by $Y=Y_0\cup Y_1\cup Y_2\cup Y_3$, with repetition, where $Y_0$ is an arbitrary set of cardinality $C\sqrt{n/k}$ and each $(Y_i)_{i=1}^{3}$ is a $p$-random subset of $V(H^k)$, with $p=C/\sqrt{nk}$. In each case we define certain properties we wish from the sets $Y_i$ and show that these properties hold with high probability.  We then show that, provided the $Y_i$ satisfy these properties then there is a winning strategy for the cops.  This is sufficient, as it is also a high probability event (by Chernoff's inequality) that each $|Y_i|\le 2C\sqrt{n/k}$, $i\in\{1,2,3\}$, and so there must exist a choice of $Y=Y_0\cup Y_1\cup Y_2\cup Y_3$ which has a winning strategy and has $|Y_0|+|Y_1|+|Y_2|+|Y_3|\le 7C\sqrt{n/k}$.
    %\sqrt{\frac{n}{k}}\cdot\phi(n)$ where $\phi$ can be any monotone non-increasing or non-decresasing funcition. By the definition of $r$, $d_V^{r+1}=\sqrt{nk}\cdot d_E\cdot \phi(n)$  

    \textbf{Case I:} $d^{r+1}\geq\sqrt{nk}\log n$ and $d^r\leq\sqrt{n/k}$.   
    \begin{proof}
    	We begin by declaring the required properties for this case.  In fact we only require information about $Y_1$.  We say that  a vertex $v$ is $j$\emph{-vertex-covered} by a set $T\subset V(H^k)$ if there is an injective function $f:S_V(v,j-1)\to T$ such that $d(u,f(u))\leq j$ for every $u\in S_V(v,j-1)$. %It \emph{almost} $r$-vertex-covers if $d(u,f(u))\leq r+1$ instead.

    	\textbf{Necessary Property of $Y_1$:} Every vertex $v\in V(H^k)$ is $(r+1)$-vertex-covered by $Y_1$.
    	
        We show that $Y_1$ has such property with probability at least $1-n^{-2}$. Let $v\in V(H^k)$ be any vertex. Let $X_v:=S_V(v,r)$. Since $H^k$ is $(d,\alpha)$-expanding, we have $|X_v|\leq \alpha^{-1}\sqrt{n/k}$ by Property \ref{proper:2} with $A=\{v\}$. Observe that $d^{r+1}\geq\sqrt{nk}\log n$. Hence, we may apply Lemma \ref{lem:vertices} to $X_v$. By the lemma, there is a probability of at least $1-n^{-3}$ that there is an injective function $f_v:X_v\to Y_1$ with $d(x,f(x))\leq r+1$. In particular with probability of at least $1-n^{-2}$, every $v\in V(H^k)$ has an injective function $f_v:X_v\to Y_1$ with $d(x,f_v(x))\leq r+1$.%, i.e. $Y_1$ $(r+1)$-vertex-covers all vertices $v\in V(H^k)$.
        
        We now show that for any set $Y_1$ with the above property, and for any choice of the starting point $v$ of the robber, there exists a strategy for the cops in $Y_0\cup Y_1$ to capture the robber.
        
        At the beginning of the game, the Cop Player will place one cop piece on each vertex of $Y_1$ and of $Y_0$ (there may be two cops pieces on the same vertex) and then the Robber Player starts with her piece at some vertex $v$. Since $v$ is $(r+1)$-vertex-covered by $Y_1$, the Cop Player may look at the image of $f_v:S_V(v,r)\to Y_1$ and send $f(u)$ to $u\in S_V(v,r)$, for each $u$, by its shortest path. As the cops move first, the robber will be trapped inside $N_V(v,r)$ by round $r+1$. After that, the cops on $Y_0$ will start to move and cover $N_V(v,r)$. As $|Y_0|\geq|N_V(v,r)|$ by Property \ref{proper:2}, the robber will be caught.
    \end{proof}
    
 \textbf{Case II:} $d^{r+1}=\sqrt{nk}\cdot\omega$, for some $\omega\in[1,\log n]$.
	
		%Let $Y_2$ and $Y_3$ be two subsets of $V(H^k)$ with the same distribution as $Y$.
		Assume $r\geq1$, we deal with the case $r=0$ in the end of this case. 
		
		For the proof of this case we will need a more involved strategy. We begin with a sketch of the proof.
		
		\textbf{Sketch:} Suppose that the robber starts at a vertex $v\in V(H^k)$. As $d^{r+1}$ is not  bigger than $\sqrt{nk}\cdot \log n$, we cannot apply Lemma \ref{lem:vertices}, as we did in Case I, to find a different cop at distance at most $r+1$ for \emph{every} vertex of $S_V(v,r)$. However, we still hope to cover \emph{most} vertices $u\in S_V(v,r)$ in this way. Property \ref{proper:4} gives us a family of disjoint sets $W(u)$ and with size $\Theta(d^{r+1})$. Each $W(u)$ is \emph{likely} to have at least one cop of $Y_1$. We consider the set $U$ of \emph{left uncovered} vertices by the cop pieces from $Y_1$, i.e., $U:=\{u\in S_V(v,r): W(u)\cap Y_1=\emptyset\}$. For each $z\in S_V(v,\lfloor r/2\rfloor)$, let $U_z:=U\cap S_V(z,\left\lceil r/2\right\rceil)$. If some $U_z$ is particularly large, we may worry that the robber will attempt to escape via $z$ and some uncovered vertex. For this reason we require to have at least $1-\varepsilon$ proportion of each $U_z$ to be covered by $Y_1$. Finally, when the robber arrives at some $z\in S_V(v,\lfloor r/2\rfloor)$ the Cop Player sends the cops in $Y_3$ to cover the $(\lfloor \frac{r}{2}\rfloor+1)$-edge spheres of each vertex of $U_z$ which will cause the robber to be trapped inside one of such spheres in the end of the strategy. $Y_2$ is only there to force the robber to move to $S_V(v,r)$ by round $r$ and end inside one of these spheres, which will be covered by team $Y_0$ at the end of the pursuit.
		
		Now, to the definitions. We recall that a vertex $v$ is $j$\emph{-vertex-covered} by a set $T\subset V(H^k)$ if there is an injective function $f_v:S_V(v,j-1)\to T$ such that $d(u,f(u))\leq j$ for every $u\in S_V(v,j-1)$. Recall from the sketch that we need not only to cover most $u\in S_V(v,j-1)$ but also most $u \in S_V(v,r)\cap S_V(z,\lceil r/2\rceil)$ for all $z\in S_V(v,\lfloor r/2\rfloor)$. Hence, we say that a vertex $v$ is $(1-\varepsilon,i,j)$\emph{-vertex-covered} by the set $T$ if there exists an injective function $f_v:S_V(v,j-1)\to T$ such that for every $z\in S_V(v,i)$ at least a $1-\varepsilon$ proportion of $u\in S_V(v,j-1)\cap S_V(z,j-1-i)$ have $d(u,f_v(u))\leq j$. Note a $(1,0,j)$-vertex-cover of $v$ is equivalent to a $j$-vertex-cover of it.
		
		Given $\Gamma\subset V(H^k)$ such that every vertex $v$ is $(1-\varepsilon,i,j)$-vertex-covered by it, let, for each $v$, $U(v,\Gamma,\varepsilon,i,j)$ be the vertices $u\in S_V(v,j-1)$ with $d(u,f_v(u))>j$. We say that $T$ is $\Gamma$\emph{-edge-good} if for each $v$ there is an injective function $f:\bigcup_{u\in U(v,\Gamma,\varepsilon,i,j)}S_E(u,j+1/2)\to T$ such that $d(e,f(e))\leq (j+1)+1/2$ for every eligible $e$.
		
		Finally, we say that a vertex $v$ is $j$\emph{-edge-covered} by a set $T\subset V(H^k)$ if there is an injective function $f_v:S_E(v,j-1/2)\to T$ such that $d(e,f_v(e))\leq j+1/2$ for every $e\in S_E(v,j-1/2)$. It is \emph{almost} $j$-edge-covered by $T$ if $d(e,f_v(e))\leq (j+1)+1/2$ instead. %Also, we say that a set $Z^\prime\subset V(H^k)$ $r$\emph{-densely-edge-covers} a vertex $v$ if there is an injective function $f:N_E(v,r-1/2)\to Z$ such that $d(e,f(e))\leq r$ for every $e\in N_E(v,r-1/2)$.
		
		We are now able to state and prove the properties of the sets $Y_i$.

		\textbf{Necessary Property of $Y_1$:} Every vertex $v\in V(H^k)$ is $(1-\frac{1}{5\omega},\lfloor r/2\rfloor,r+1)$-vertex-covered by $Y_1$.
		
		Let $v\in V(H^k)$ be any vertex. Since $H^k$ is $(d,\alpha)$-expanding, Property \ref{proper:4} gives a family $\mathcal{F}=\{W(u)\subset N_V(u,r+1):u\in S_V(v,r)\}$ of disjoint subsets, each with $|W(u)|\geq\alpha d^{r+1}$. For each $u\in S_V(v,r)$, the probability of $W(u)$ not containing a vertex of $Y_1$ is
		$$\left(1-\frac{C}{\sqrt{nk}}\right)^{|W(u)|} \,\leq\, \exp\left(-\frac{C}{\sqrt{nk}} \cdot\alpha d^{r+1}\right) \,=\, \exp\left(-C\omega\alpha\right) \,\leq\, \frac{\alpha}{10\omega}$$%<\frac{1}{10\omega}$$
		since $d^{r+1}=\sqrt{nk}\cdot \omega$ and $C\geq 10\alpha^{-2}$.% Also, since  $e^{x/(x-1)}\leq (1-x)$ we have
		%$$\left(1-\frac{C}{\sqrt{nk}}\right)^{|W(u)|} \,\geq\, \exp\left(-\frac{C}{\sqrt{nk}-C} \cdot\alpha d^{r+1}\right) \,\geq\, \exp\left(-2C\omega\alpha\right).$$
		
		For some fixed $z\in S_V(v,\lfloor r/2\rfloor)$, let $U_z$ be the set of vertices $u$ in $S_V(v,r)\cap S_V(z,\left\lceil r/2\right\rceil)$ for which $W(u)\cap Y_1=\emptyset$. %Since $H^k$ is $(d,\alpha)$-expanding, by Property \ref{proper:2} with $A=\{v\}$ we have $\alpha d^{\left\lceil r/2\right\rceil}\leq |S_V(z,\left\lceil r/2\right\rceil)|\leq \alpha^{-1}d^{\left\lceil r/2\right\rceil}$. 
		Therefore,
		$$\mathbb{E}[|U_z|] \,\leq\, \big|S_V(z,\left\lceil r/2\right\rceil)\big|\cdot(\alpha/10\omega).$$% \,\leq\, d^{\left\lceil \frac{r}{2}\right\rceil}/10\omega$$
		%and
		%$$\mathbb{E}[|U_z|] \,\geq\, \big|S_V(z,\left\lceil r/2\right\rceil)\big|\cdot(\alpha/10\omega) \,=\, \Omega(d^{\left\lceil \frac{r}{2}\right\rceil}/\omega).$$
		
		Note also that since the sets $W(u)$ are disjoint, then each event $W(u)\cap Y_1=\emptyset$ is independent with probability bounded by $\alpha/10\omega$. Hence, $|U_z|$ is stochastically dominated by $\Bin\big(\big|S_V(z,\left\lceil r/2\right\rceil)\big|,\alpha/10\omega\big)$. By Chernoff's Inequality, Theorem \ref{thm:chernoff}, we get that
			\begin{align*}
				\mathbb{P}\left(|U_z|\geq \frac{2\alpha}{10\omega}\big|S_V(z,\lceil r/2\rceil)\big|\right)&\,\leq\, \exp\left(- \frac{\alpha\big|S_V(z,\left\lceil r/2\right\rceil)\big|}{10\omega}\right) \\
				&\,\leq\, \exp\left(-\frac{\alpha^2d^{\lceil r/2\rceil}}{10\omega}\right)\\
				&\,\leq\, n^{-3},
			\end{align*}
		%$$\mathbb{P}\left(|U_z|\geq \frac{2\alpha}{10\omega}\big|S_V(z,\lceil r/2\rceil)\big|\right) \,\leq\, \exp\left(-\Omega\left(\frac{d^{\lceil r/2\rceil}}{\omega}\right)\right) \,\leq\, n^{-3},$$
		since we have $\big|S_V(z,\left\lceil r/2\right\rceil)\big|\geq \alpha d^{\lceil r/2\rceil}$ by Property \ref{proper:2} of $(d,\alpha)$-expanding, and $d^{\lceil r/2\rceil}\geq d^1> \log^6n$ (as $r\geq 1$).%  and so the probabilities are independent), $|U|\leq 2^{-1}\alpha^{-1}d_V^{\left\lceil r/2\right\rceil}\cdot\exp(-C\omega/2)$ with probability at least $1-o(n^{-2})$. 
		
		Hence, with probability at least $1-n^{-3}$ the proportion of uncovered vertices of $U_z$ is at most $1/5\omega$. Then, by a union bound over $z\in S_V(v,\lfloor r/2\rfloor)$ the vertex $v$ is $(1-1/5\omega,\lfloor r/2\rfloor, r+1)$-vertex-covered by $Y_1$ with probability at least $1-n^{-2}$. In particular, with probability at least $1-n^{-1}$ this property holds for every $v\in V(H^k)$.
		
		\textbf{Necessary Property of $Y_2$:} All vertices $v\in V(H^k)$ are almost $r$-edge-covered by $Y_2$.
		
		For a given vertex $v\in V(H^k)$ we would like to apply Lemma \ref{lem:edges} to  $X_v:=S_E(v,r-1/2)$. Note that since $d^r\leq \sqrt{nk}$, we may observe that Property \ref{proper:1} guarantees that $|S_E(v,r-1/2)|\leq \alpha^{-1}\sqrt{n/k}$, as $d^r/k\leq\sqrt{n/k}$. Also note that 
		$$d^{r+1}=\sqrt{nk}\cdot\omega=k\sqrt{\frac{n}{k}}\cdot\omega\geq \sqrt{\frac{n}{k}}\log n,$$ since $k\geq\log^3 n$. 
		
		As $Y_2$ is a $\frac{C}{\sqrt{nk}}$-random subset of $V(H^k)$, we may apply Lemma \ref{lem:edges} to $X_v$. By the lemma, with probability at least $1-n^{-3}$ there is an injective function $f_v:X_v\to Y_2$ with $d(e,f(e))\leq (r+1)+1/2$. In particular with probability at least $1-n^{-2}$, all vertices $v\in V(H^k)$ are almost $r$-edge-covered by $Y_2$.

		\textbf{Necessary Property of $Y_3$:} $Y_3$ is $Y_1$-edge-good.
		
		We may assume $Y_1$ satisfies its condition. That is, let $Y_1$ be a set that $(1-1/5\omega,\lfloor r/2\rfloor,r+1)$-vertex-covers every vertex $v\in V(H^k)$. Recall that we defined $U_z\subset S_V(v,r)\cap S_V(z,\left\lceil r/2\right\rceil)$ as the random set of vertices in this intersection such that $W(u)\cap Y_1=\emptyset$, which now is deterministic. Let $S:=\bigcup_{u\in U_z}S_E(u,\left\lfloor \frac{r}{2}\right\rfloor+\frac{1}{2})$. Since $r\geq 1$, we have $d^{\left\lfloor r/2\right\rfloor+1}\leq \sqrt{nk}$. Hence, by Property \ref{proper:1} we have that $|S_E\left(s,\left\lfloor \frac{r}{2}\right\rfloor+\frac{1}{2}\right)|\leq\frac{d^{\left\lfloor r/2\right\rfloor+1}}{\alpha k}$, for any $s\in V$. And it follows that
		\begin{align*}
			|S|\leq |U_z|\cdot \max_{s\in U_z}\left|S_E\left(s,\left\lfloor \frac{r}{2}\right\rfloor+\frac{1}{2}\right)\right|\,&\leq\, \frac{d^{\left\lceil \frac{r}{2}\right\rceil}}{5\omega}\cdot \frac{d^{\left\lfloor \frac{r}{2}\right\rfloor+1}}{\alpha k} \\
			&=\, \frac{\sqrt{nk}\cdot\omega}{5k\alpha\omega}\\
			&=\, \frac{1}{5\alpha}\sqrt{\frac{n}{k}}.
		\end{align*}
		
		Finally, recall that $d^{r+1}\geq\sqrt{\frac{n}{k}}\log n$. Hence, we may apply Lemma \ref{lem:edges} to $X=S$ with $Y_3$. By the lemma, with probability at least $1-n^{-3}$ there is an injetive function $f_{v,z}:S\to Y_3$ such that $d(e,f(e))\leq r+1/2$. In particular with probability at least $1-n^{-1}$, a union bound over the pairs $(v,z)$ shows the existence of an $f_{v,z}$ for each pair.

		%Before we proceed, there is a catch during this pursuit: the cops must force the robber to go to $S_V(v,r)$ during round $r$ for the next step of our strategy to work. For such, we recall $d_V^k/k<\sqrt{n/k}$ and note that $d_V^{r+1}=k\sqrt{\frac{n}{k}}\cdot\omega\gg \sqrt{\frac{n}{k}}\log n$, since $k=\omega(\log^3 n)$, so we can apply Lemma \ref{lem:edges} with $X_1=N_E(v,r-\frac{1}{2})$ ($|X_1|\leq c^{-1}d_V^r/k$ by property (1) of expansion) and $d_V^{r+1}$ instead of $d_V^r$. This must be done with a second team of cops which will be placed, at the beginning of the game, on the vertices of the set $Y_1$ generated by the lemma and go directly to their match defined by injection $f:X_1\to Y_1$. Therefore, the robber must be in $S_V(v,r)$ during round $r$ if she wants to escape.
		
		With the necessary properties all defined, we are now ready to describe the strategy. The Cop Player starts by placing a cop piece on each vertex of $Y_1$, $Y_2$, $Y_3$ and $Y_0$ with repetition if they intersect. Then the Robber Player places her piece at some vertex $v\in V(H^k)$. Since $v$ is $(1-\frac{1}{5\omega},\lfloor r/2\rfloor,r+1)$-vertex-covered by $Y_1$, from the start of the game the Cop Player moves all of his cop pieces from $Y_1$ that are in some $W(u)$ by their shortest path towards $u$. This leaves a $\frac{1}{5\omega}$ proportion of the vertices from $S_V(v,r)$ to be used as possible escape routes by the Robber Player. At the same time, since $v$ is almost $r$-edge-covered by $Y_2$, we have an injective function $f_v:S_E(v,r-1/2)\to Y_2$ such that $d(e,f_v(e))\leq r+3/2$. The Cop Player will move its pieces from $f_v(e)$ to the closest vertex of $e$, for all $e\in S_E(v,r-1/2)$, by their shortest path. This team serves only to force the Robber Player to run directly to $S_V(v,r)$ by round $r$, otherwise she will be trapped inside of $S_E(v,r-1/2)$ by round $r+1$ by team $Y_2$.
		
		Starting on round $\left\lfloor \frac{r}{2}\right\rfloor$, the Cop Player moves the cops of team $Y_3$. Since $Y_3$ is $Y_1$-edge-good, he may observe the current position of the robber, say $z$, and each $u\in Y_3$ which is in the image of $f_{v,z}$. He then, again, moves each cop piece from $u$ to $f^{-1}_{v,z}(u)$ using their shortest path.
		
		Since team $Y_2$ forced the Robber Player to try to escape from its $r$th-neighbourhood, she can't return and, hence, must continue her original trajectory defined by $z$. Therefore, after team $Y_3$ finishes its movement the robber piece will be trapped inside one of the spheres defined by $U_z$ or inside the sphere $S_V(v,r)$. Finally, team $Y_0$ is sent to occupy all of the vertices of such sphere, which is possible since $|Y_0|=C\sqrt{n/k}$ (bigger than $N_V(v,r-1)$ and all of the spheres defined by $z$), catching the robber.
		
		Now we deal with $r=0$. Note that here $d=\sqrt{nk}\cdot\omega$. Suppose that the robber starts at a vertex $v\in V(H^k)$. As $S_V(v,1)$ is large, we cannot cover it with $\Theta(\sqrt{n/k})$ cops. For this reason the first team of cops, $Y_1$ attempt to cover $S_E(v,1/2)$ instead. For each $e\in S_E(v,1/2)$, the probability that $S_V(e,3/2)$ contains a vertex from $Y_1$ is at most $\exp(-C\omega \alpha)$. Hence, the Cop Player can, with high probability, cover all but a $(1/\omega)$ proportion of $S_E(v,1/2)$ within one move. Let $D\subset S_E(v,1/2)$ be the set of the edges from $S_E(v,1/2)$ that cannot be covered by $Y_1$ in one move. As $|D|\leq c\sqrt{n/k}$, for some $c>0$, and $d=\sqrt{nk}\cdot\omega=k\sqrt{n/k}\cdot\omega> \sqrt{n/k}\cdot \log n$, we may apply Lemma \ref{lem:edges} to $D$ with $Y_3$ and cover it in the first turn as well. Since all of $S_E(v,1/2)$ is now covered, the robber can't escape from it and in two more turns the Cop Player may win.%We would like to have a big proportion of $S:=S_E(v,1/2)$ within distance $3/2$ of $Y_1$. For each edge $e$ we have $|N_V(e,3/2)|\geq\alpha\min\{k\sqrt{nk}\omega,n\}$ by Property \ref{proper:2}. Hence, for each $e$ that contains $v$, the probability of $N_V(e,3/2)\cap Y_1=\emptyset$ is bounded by 

		The next two cases are complementary to the first two, switching, where needed, the vertex-surrounding strategy by the edge-surrounding strategy and vice-versa.

	\textbf{Case III:} $d^r\geq\sqrt{n/k}\log n$.

		\textbf{Necessary Property of $Y_1$:} Every vertex $v\in V(H^k)$ is $r$-edge-covered by $Y_1$.
		
		Let $v\in V(H^k)$ be any vertex. Define $X_v:=S_E(v,r-1/2)$. Since $H^k$ is $(d,\alpha)$-expanding, we have by Property \ref{proper:1} that $|X_v|< \alpha^{-1}\sqrt{n/k}$. Also, observe that $d^{r}\geq\sqrt{n/k}\log n$. Hence, we may apply Lemma \ref{lem:edges} to $X_v$. By the lemma, there is a probability of at least $1-n^{-3}$ that there is an injective function $f_v:X_v\to Y_1$ with $d(e,f(e))\leq r+1/2$. In particular with probability at least $1-n^{-2}$, for each vertex $v\in V(H^k)$ there is an injective function $f_v:X_v\to Y$ with $d(e,f(e))\leq r+1/2$ for all $v\in V(H^k)$.
		
		As in Case I, the Cop Player simply places a cop piece on each vertex of $Y_0$ and $Y_1$, with repetition in the intersection, and then observes where the robber starts and move the cops directly towards their matching defined by the corresponding function. In round $r$ every edge of $S_E(v,r-1/2)$ will be occupied and team $Y_0$ finish the pursuit to catch the robber, since $|S_V(v,r-1)|\leq |X_v|\leq |Y_0|$.

	\textbf{Case IV:} $d^{r+1}\geq\sqrt{nk}\log n$ and $d^r=\sqrt{n/k}\cdot\omega$, $1\leq \omega\leq\log n$.
	
		The proof of this case will be similar to Case 2 and hence we need two more definitions, as teams $Y_1$ and $Y_3$ will switch from surrounding vertices to surrounding edges and vice-versa.
		
		Recall that we say that a vertex $v$ is $j$\emph{-edge-covered} by a set $T\subset V(H^k)$ if there is an injective function $f_v:S_E(v,j-1/2)\to T$ such that $d(e,f_v(e))\leq j+1/2$ for every $e\in S_E(v,j-1/2)$. It is \emph{almost} $j$-edge-covered by $T$ if $d(e,f_v(e))\leq (j+1)+1/2$ instead.
		
		We say that a vertex $v$ is $(1-\varepsilon,i,j)$\emph{-edge-covered} by a set $T\subset V(H^k)$ if there is an injective function $f_v:S_E(v,j-1/2)\to T$ such that for every $z\in S_V(v,i)$ at least a $1-\varepsilon$ proportion of the edges $e\in S_E(v,j-1/2)\cap S_E(z,j-i-1/2)$ has $d(e,f(e))\leq j+1/2$. Note that a set that $(1,0,j)$-edge-covers $v$ is equivalent to a set $j$-edge-covering it.
		
		Furthermore, given $\Gamma\subset V(H^k)$ such that every vertex $v$ is $(1-\varepsilon,i,j)$-edge-covered by it, let, for each $v$, $U(v,\Gamma)$ be the edges with $d(e,f_v(e))>j+1/2$. We say that $T$ is $\Gamma$\emph{-vertex-good} if for each $v$ there is an injective function $f:\bigcup_{e\in U(v,\Gamma)}S_V(u,j)\to T$ such that $d(u,f(u))\leq (j+1)+1$ for every eligible $u$.
		
		Note that here $r\not=0$ since $d^0=\sqrt{n/k}\cdot\omega$ means $k\geq n$, but $n\geq d=\omega(k\log^3n)$.
		
		\textbf{Necessary Property of $Y_1$:} Every vertex $v\in V(H^k)$ is $(1-1/5\omega,\lfloor r/2\rfloor,r)$-edge-covered by $Y_1$.
		
		Let $v\in V(H^k)$ be any vertex. Since $H^k$ is $(d,\alpha)$-expanding, Property \ref{proper:5} gives a family $\mathcal{F}=\{W(e)\subset N_V(e,r+1/2):e\in S_E(v,r)-1/2\}$ of disjoint subsets, each with $|W(u)|\geq \alpha kd^{r}$. For each $e\in S_E(v,r-1/2)$, the probability of $W(e)$ not having a vertex of $Y_1$ is
		$$\left(1-\frac{C}{\sqrt{nk}}\right)^{|W(e)|} \,\leq\, \exp\left(-\frac{C}{\sqrt{nk}} \cdot \alpha kd^{r}\right) \,=\, \exp\left(-C\omega\alpha\right) \,\leq\, \frac{\alpha}{10\omega}.$$
		since $d^r=\sqrt{\frac{n}{k}}\cdot\omega$ and $C\geq 10\alpha^{-2}$.
		
		For $z\in S_V(v,\lfloor r/2\rfloor)$, let $U_z\subset S_E(z,\left\lceil \frac{r}{2}\right\rceil-\frac{1}{2})\cap S_E(v,r-\frac{1}{2})$ be the set of edges in this intersection for which $W(e)\cap Y_1=\emptyset$%. Since $H^k$ is $(d,\alpha)$-expanding and $d^r\leq \sqrt{nk}$%, by Property \ref{proper:1} we have that $|S_E(z,\left\lceil \frac{r}{2}\right\rceil-\frac{1}{2})| \leq \alpha^{-1}d^{\left\lceil \frac{r}{2}\right\rceil}/k$
		. Then
		$$\mathbb{E}[|U_z|]\leq \left|S_E\left(z,\left\lceil \frac{r}{2}\right\rceil-\frac{1}{2}\right)\right|\cdot\frac{\alpha}{10\omega} .$$%\leq \frac{d^{\left\lceil\frac{r}{2}\right\rceil}}{10k\omega}.$$
		
		Recall that the sets $W(e)$ are all disjoint and hence each event $W(e)\cap Y=\emptyset$ is independent with probability bounded from above by $\alpha/10\omega$. By Chernoff's Inequality, we get that
		$$\mathbb{P}\left(|U_z|\geq \frac{2}{10\omega}\left|S_E\left(z,\left\lceil\frac{r}{2}\right\rceil-\frac{1}{2}\right)\right|\right) \leq \exp\left(-\Omega\left(\frac{d^{\left\lceil\frac{r}{2}\right\rceil}}{k\omega}\right)\right) \leq n^{-3},$$
		using that $\big|S_E(z,\left\lceil \frac{r}{2}\right\rceil-\frac{1}{2})\big|\geq \alpha d^{\lceil r/2\rceil}/k$ by Property \ref{proper:1} of $(d,\alpha)$-expanding gives, and since $d^{\lceil r/2\rceil}/k\geq d^1/k> \log^3n$ (as $r\geq 1)$.
		
		Therefore, with probability at least $1-n^{-3}$ the proportion of uncovered edges of $U_z$ is $1/5\omega$. In particular, with probability at least $1-n^{-1}$, this property holds for every $v\in V(H^k)$ and every $z\in S_V(v,\lfloor r/2\rfloor)$, i.e., every vertex $v\in V(H^k)$ is $(1-1/5\omega,\lfloor r/2\rfloor,r)$-edge-covered by $Y_1$.

		\textbf{Necessary Property of $Y_2$:} All vertices $v\in V(H^k)$ are almost $r$-edge-covered $Y_2$.

		The proof of this property is almost the same as in Case II and will be omitted. In fact, the only difference is that we need that $d/k\geq \log^3n$ rather than $k\geq \log^3n$.
		
		\textbf{Necessary Property of $Y_3$:} $Y_3$ is $Y_1$-vertex-good.
		
		We may assume $Y_1$ satisfies its condition. That is, let $Y_1$ be a set that $(1-1/5\omega,\lfloor r/2\rfloor,r)$-edge-covers all vertices $v\in V(H^k)$. Recall that we defined $U_z\subset S_E(z,\left\lceil \frac{r}{2}\right\rceil-\frac{1}{2})\cap S_E(v,r-\frac{1}{2})$ as the random set of edges in this intersection such that $W(e)\cap Y_1=\emptyset$, which now is deterministic. Let $S:=\bigcup_{e\in U_z}S_V(e,\left\lfloor\frac{r}{2}\right\rfloor+\frac{1}{2})$. By Property \ref{proper:2} applied to $e$ we have that $|S_V(e,\left\lfloor\frac{r}{2}\right\rfloor+\frac{1}{2})|\leq \alpha^{-1}kd^{\lfloor r/2\rfloor}$. Since $d^r=\sqrt{n/k}\cdot \omega$, it follows that
		$$|S|\,\leq\, |U_z|\cdot \left|S_V\left(e,\left\lfloor \frac{r}{2}\right\rfloor +\frac{1}{2}\right)\right| \,\leq\, \frac{d^{\left\lfloor \frac{r}{2}\right\rfloor}}{5k\omega}\cdot \frac{kd^{\left\lceil \frac{r}{2}\right\rceil}}{\alpha} \,\leq\, \frac{1}{5\alpha}\sqrt{\frac{n}{k}}.$$
		
		Finally, recall that $d^{r+1}\geq\sqrt{nk}\log n$. Hence we can use Lemma \ref{lem:vertices} to $X=S$ with $Y_3$. By the lemma, with probability at least $1-n^{-3}$ there is an injective function $f_{v,z}:S\to Y_3$ with $d(e,f(e))\leq r+1/2$. In particular with probability at least $1-n^{-1}$, this property holds for every $v\in V(H^k)$ and $z\in S_V(v,\lceil r/2\rceil)$. 
		
		We are now ready to state the strategy. The Cop Player starts by placing a cop on each vertex of $Y_1$, $Y_2$, $Y_3$ and $Y_0$ with repetition if they intersect. Then the Robber Player decides to start at some vertex $v\in V(H^k)$. A cop piece in a vertex of $Y_1$ that is in some $W(e)$ moves to $e$ by its shortest path towards $e$. A cop piece from team $Y_2$ will be moved according to $f_v:S_E(v,r-1/2)\to Y_2$. As in Case II, this team serves only to force the Robber Player to run directly to $S_V(v,r)$ by turn $r$, otherwise it will be trapped inside of $S_E(v,r-1/2)$ by turn $r+1$ by team $Y_2$.
		
		Team $Y_3$ begins moving in round $\left\lfloor \frac{r}{2}\right\rfloor$ according to the function $f_{v,z}$ and trap the robber piece inside one of the vertex-spheres defined by $U_z$ or inside the sphere $S_E(v,r-1/2)$. Team $Y_0$ then proceeds to occupy every vertex of such sphere, since $|Y_0|$ is bigger than all of the spheres (the biggest one being $S_V(v,r-1)$ with size bounded by $\alpha d^{r-1}$ or $1$, depending on $r$), catching the robber.

	This ends the proof showing that the Cop Player needs at most $7C\sqrt{n/k}$ cops to have a winning strategy.
\end{proof}

\section{Random Hypergraphs are $(d,\alpha)$-Expanding}\label{sec:expanding}
 
 Now that we showed that $k$-uniform hypergraphs that are $(d,\alpha)$-expanding have cop number at most $O(\sqrt{n/k})$ we just need to show that $H^k(n,p)$, for our ranges of $k$ and $p$, really is $(d,\alpha)$-expanding with high probability. Properties \ref{proper:1}-\ref{proper:3} of Definition \ref{def:exp} have already been proved to hold with high probability.

\begin{theorem}[Theorem 1.7 of \cite{HyperLog}]\label{thm:hyperlog-exp}
	There exists a universal constant $\alpha>0$ such that if $k=k(n)$, $p=p(n)>0$ are such that $k=\omega(\log n)$ and $\frac{n}{k}\geq p\binom{n-1}{k-1}=\omega(\log^3n)$, then, with high probability, $H^k(n,p)$ has Properties \ref{proper:1}-\ref{proper:3} of Definition \ref{def:exp} for $d=\hat{d}=kp\binom{n-1}{k-1}$ .
\end{theorem}

We remark that although the universal constant $\alpha$ that they found is small, they managed to find better constants for the vertex-expansion if you only allow the use of smaller sets and/or neighbourhoods.

\begin{prop}[Equation 4.10 of \cite{HyperLog}]\label{prop:hyperlog-exp}
	Let $k=k(n)$, $p=p(n)>0$ be such that $k=\omega(\log n)$ and $\frac{n}{k}\geq p\binom{n-1}{k-1}=\omega(\log^3n)$. Then, for every $A\subset V$ with $|A|=a$ and every $r\in\mathbb{N}$ such that $a\hat{d}^r\leq\frac{n}{2\log n}$, we have
	\begin{equation}\label{eq:remark}
	2^{-5}a\hat{d}^r\leq|N_V(A,r)|\leq 2a\hat{d}^r
	\end{equation}
	with high probability.
\end{prop}

We are now ready to show that random hypergraphs are $(\hat{d},\alpha)$-expanding with high probability.

\begin{prop}\label{prop}
	Let $k\geq \log^3 n$, $\frac{n}{k}\geq p\binom{n-1}{k-1}=\omega(\log^3 n)$ and $\hat{d}=kp\binom{n-1}{k-1}$. Then there is a universal constant $\alpha>0$ such that the following holds for $H^k(n,p)$ with high probability:
	
	\begin{enumerate}
		\setcounter{enumi}{3}
		\item Let $r\in\mathbb{N}$ be such that $\sqrt{nk}<\hat{d}^{r+1}\leq\sqrt{nk}\log n$. Then, for all $v\in V(H^k(n,p))$ there exists a family 
		$$\big\{W(u)\subset N_V(u,r+1)\,:\,u\in S_V(v,r)\big\}$$ of pairwise disjoint subsets such that, for each $u\in S_V(v,r)$ 
		$$|W(u)|\geq \alpha \hat{d}^{r+1};$$
		\item Let $r\in\mathbb{N}$ be such that $\sqrt{\frac{n}{k}}<\hat{d}^{r}\leq\sqrt{\frac{n}{k}}\log n$. Then, for all $v\in V(H^k(n,p))$ there exists a family 
		$$\left\{W(e)\subset N_V\left(e,r+\frac{1}{2}\right)\,:\,e\in S_E\left(v,r-\frac{1}{2}\right)\right\}$$ of pairwise disjoint subsets such that, for each $e\in S_E(v,r-1/2)$ 
		$$|W(e)|\geq \alpha k\hat{d}^{r}.$$
	\end{enumerate}
\end{prop}
\begin{proof}
	By Theorem \ref{thm:hyperlog-exp} and Proposition \ref{prop:hyperlog-exp} we may assume that Properties \ref{proper:1}-\ref{proper:3} and Equation \ref{eq:remark} hold deterministically. We begin by proving (4). Fix $v\in V(H^k(n,p))$. We want to apply Proposition \ref{prop:hyperlog-exp} to $\bigcup_{i\leq |S_V(v,r)|}N_V(u_i,r+1)$.
	
	Note that since $\sqrt{nk}<\hat{d}^{r+1}\leq\sqrt{nk}\log n$ and that $\hat{d}\geq \log^6n$, then $r$ is unique. Let $(u_i)$ be an ordering of the vertices of $S_V(v,r)$. We now bound the size of $\bigcup_{i\leq |S_V(n,r)|}N_V(u_i,r+1)$. First, note that as $n\geq\hat{d}=\omega(k\log^3n)$, then $k\leq\frac{n}{\log^3n}$. And since 
	$$\hat{d}^r\leq\sqrt{nk}\leq\frac{n}{\log^{3/2} n},$$
	we may apply Proposition \ref{prop:hyperlog-exp} to $S_V(v,r)$ and get  
	$$|S_V(v,r)|\leq 2\hat{d}^{r}\leq \frac{2n}{\log^{3/2} n}.$$
	Finally, observe that
	
	$$|S_V(v,r)|\cdot \hat{d}^{r+1}\,\leq\, 2\hat{d}^{r}\cdot\hat{d}^{r+1}\,\leq\, \frac{2(\sqrt{nk}\log n)^2}{\hat{d}}\leq\frac{n}{2\log n}$$
	since $\hat{d}=\omega(k\log^3n)$. Hence, by Proposition \ref{prop:hyperlog-exp} we have
	$$\left|\bigcup_{i\leq |S_V(v,r)|}N_V(u_i,r+1)\right|\,\geq\, 2^{-5}\hat{d}^{r+1}|S_V(v,r)|.$$

	The same argument holds for every subset of $S_V(v,r)$. Therefore, we may apply  Theorem \ref{thm:Hall} (Hall's) with $\ell=2^{-5}\hat{d}^{r+1}$ to conclude a multiplicity matching between $S_V(v,r)$ and $\bigcup_{i\leq |S_V(v,r)|}N_V(u_i,r+1)$. That is, we may match each $u_i$ to a set $W(u_i)\subset N_V(u_i,r+1)$ with $|W(u_i)|\geq 2^{-5}\hat{d}^{r+1}$ that is disjoint from the others. As $v$ was arbitrary, this proves (4) with $\alpha=2^{-5}$.

	For (5), we first observe that Proposition \ref{prop:hyperlog-exp} cannot help us since it only deals with vertex sets. Hence we may only use Theorem \ref{thm:hyperlog-exp}. Applying Theorem \ref{thm:hyperlog-exp} to $N_V(S_E(v,r-1/2),r+1/2)$ we get
	\begin{align*}
		\big|N_V(S_E(v,r-1/2),r+1/2)\big|&\,\geq\, \alpha k\hat{d}^{r}\big|S_E(v,r-1/2)\big|.
	\end{align*} 
	
	In the same way as before, the same argument holds for all subsets of $S_E(v,r-1/2)$ and we may find a multiplicity matching by Theorem \ref{thm:Hall} with $\ell=\alpha k\hat{d}^r$, which proves (5) with the same $\alpha$ from Theorem \ref{thm:hyperlog-exp}. This completes the proof.
\end{proof}

	We are now ready to complete the proof of Theorem \ref{thm:main}. 
	
	\begin{proof}\textit{[of Theorem \ref{thm:main}]}
		Let $H^k(n,p)$ be a $k$-uniform random graph with $k\geq \log^3n$ and $n\geq \hat{d}/k=\omega(\log^3 n)$, where $\hat{d}=kp\binom{n-1}{k-1}$. Theorem \ref{thm:hyperlog-exp} together with Proposition \ref{prop} shows that $H^k(n,p)$ is with high probability $(\hat{d},\alpha)$-expanding, for some universal $\alpha>0$ given by Theorem \ref{thm:hyperlog-exp}. Also, Theorem \ref{thm:densedet} shows that any $(\hat{d},\alpha)$-expanding $k$-uniform hypergraph has cop number at most $F(\alpha)\sqrt{n/k}$. Therefore, $c(H^k(n,p))=O(\sqrt{n/k})$ with high probability.
	\end{proof}

\section{Towards the Complete Proof for Random Hypergraphs}\label{sec:remarks}
 We recall that Conjecture \ref{conj:conj} states that $c(H)=O(\sqrt{n/k})$ for all connected $k$-uniform hypergraphs, and that Theorem \ref{thm:main} confirms this for random hypergraphs with $k\geq\log^3n$ and $p\binom{n-1}{k-1}=\omega(\log^3n)$. In this section we discuss whether these conditions can be relaxed. To do so would require improvements to both our results: the deterministic $(d,\alpha)$-expanding Theorem \ref{thm:densedet} and the random hypergraph properties from Theorem \ref{thm:hyperlog-exp} and Proposition \ref{prop}.
 
 In Figure \ref{fig:results} we summarize our results from Theorem \ref{thm:densedet} visually. Note that during the proof of each case of Theorem \ref{thm:densedet} different bounds on the sizes of $k$ and $d/k$ were necessary for the arguments to hold. That is, in some cases we didn't need to use the stronger bounds $k\geq\log^3n$ and $d/k\geq \log^3n$, although they were used on at least one case each. 
 
 %Since the proof of the theorem is for deterministic $(d,\alpha)$-expanding $k$-regular hypergraphs, if one improves the bounds from Theorem \ref{thm:hyperlog-exp}, one may prove the conjecture on the random case for some remaining ranges of $k$ and $d/k$ not dealt by Theorem \ref{thm:main}.
 
 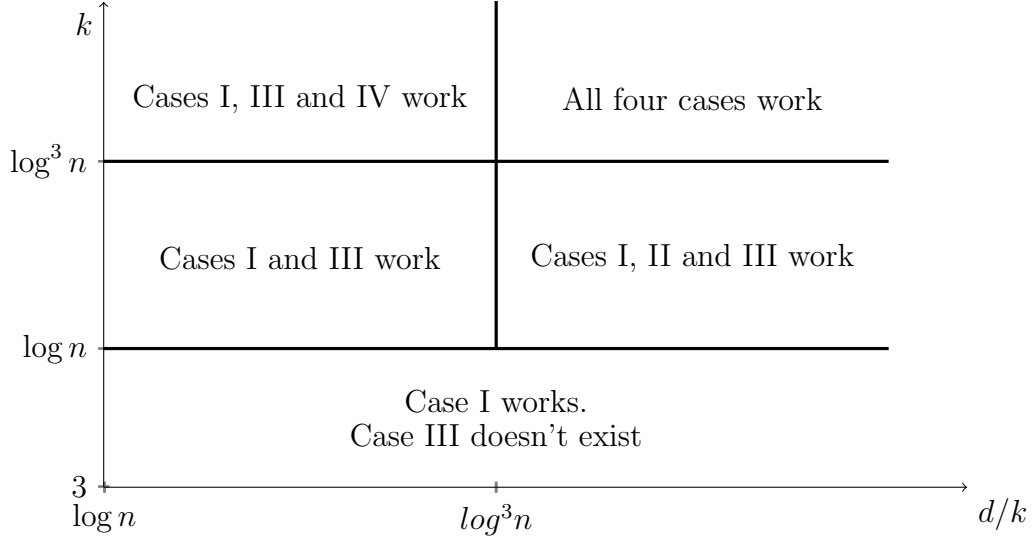
\begin{figure}[H]
 	\begin{tikzpicture}
 		
 		\begin{axis}[xmin=0,ymin=0,xmax=110,ymax=70,axis lines=middle, standard, axis line style={->},
 			xlabel=$d/k$,
 			ylabel=$k$,
 			minor xtick={0.1,50},
 			tick style={line width=1pt},
 			xtick={0.1,50},
 			xticklabels={$\vspace{1mm} \log n$, $log^3 n$},
 			%minor ytick={0,5,10,1,25.2},
 			ytick={0.1,20,47},
 			yticklabels={3,$\log n$, $\log^3 n$}]
 			
 			\draw (25,33) node {Cases I and III work};
 			\draw (50,12.5) node {Case I works.};
 			\draw (50,7.5) node {Case III doesn't exist};
 			\draw (75,33) node {Cases I, II and III work};
 			\draw (25,56) node {Cases I, III and IV work};
 			\draw (75,56) node {All four cases work};
 			
 			\addplot[name path=GGSS,very thick] coordinates {(0,20) (100,20)};
 			\addplot[name path=GGS,very thick] coordinates {(0,47) (100,47)};
 			\addplot[name path=NRS,very thick] coordinates {(50,20) (50,680)};
 			
 		\end{axis}
 		
 	\end{tikzpicture}
 	\caption{In this figure we show a graph representing the ranges for which we've shown Conjecture \ref{conj:conj} to hold for $(d,\alpha)$-expanding $k$-uniform hypergraphs.}
 	\label{fig:results}	
 \end{figure}

 %The figure also shows which are the few remaining open cases of (the hypergraph generalization of) Meyniel’s conjecture. In the proof of Theorem \ref{thm:densedet} for each case we needed to use different bounds on the sizes of $k$ and $d/k$. In Cases I and III we didn't use any, although Case 3 colapses if $k\leq \log n$. But on Case II it was necessary that $k\geq \log n$ and that $d/k\geq \log^3 n$. Also, on Case IV we only needed that $k\geq \log^3 n$. 

In this sense, one may ask what are the \textbf{necessary} lower bounds of $k$ and $d/k$ for which the conjecture may hold. We know that $d/k$ must be of order at least $\log n$ since this is the connectivity threshold of the random $k$-uniform hypergraph, as seen in \cite{CKK}. For $k$, one may think that it could work all the way down to $k=2$ (the graph case), but one needs to show that the number of vertices in any set of edges of a $k$-uniform random hypergraph $H^k(n,p)$ is concentrated for $k\geq 3$. However, Lemma 4.2 of \cite{HyperLog} (stated bellow) only shows this concentration for $k=\omega(\log n)$, which is one of the reasons why Theorem \ref{thm:hyperlog-exp} has this necessary condition on the size of $k$. One may improve this lemma to $k=O(1)$, but there is another constrainment.

\begin{lem}[Lemma 4.2 of \cite{HyperLog}]\label{lem:improv}
    For every constant $\varepsilon\in(0,1/2]$ and for every random hypergraph $H^k(n,p)$ with $k=\omega(\log n)$ and $d\leq n$ the following holds.

    Let $B\subseteq E(H^k(n,p))$ be any subset of edges with $|B|=b$. Let $V_B:=\{v\in e:e\in B\}$ be the set of vertices in at least one edge of $B$. Then if $\left(\frac{bk}{n}\right)^\varepsilon\leq 2^{-5}$, we have
    $$|V_B|\geq(1-\varepsilon)bk$$
    for all $B$ with high probability. Moreover,
    $$|V_B|\geq2^{-12}bk.$$
\end{lem}

The original proof of Theorem \ref{thm:hyperlog-exp} uses a sequence of very small error terms $(\varepsilon_m)$ to show the concentration of the vertex-neighbourhood expansion. One of the reasons is that the neighbourhood expansion is done layer by layer. We know that for a given $v\in V(H^k(n,p))$ we have $|N_E(v,1/2)|\in[\alpha d/k,\alpha^{-1}d/k]$. Then for $\varepsilon_1>0$ we may use Lemma \ref{lem:improv} to get $|N_V(v,1)|\in [(1-\varepsilon_1)\alpha d,\alpha^{-1}d]$. Iterating this expansion process many times may give loose bounds if the sequence $(\varepsilon_i)$ is not finely tuned. In the proof of Theorem \ref{thm:hyperlog-exp} the authors defined their sequence as
$$\varepsilon_i:=\frac{5}{\log n-\log(2d^i)},$$
and hence we cannot get weaker bounds on $k$ for this theorem by only improving on Lemma \ref{lem:improv}.

We now discuss the random case. As we can see from Figure \ref{fig:results}, there are sequences of hypergraphs for which Theorem \ref{thm:main} proves that $c(H^k(n,p))=O(\sqrt{n/k})$ even though $k<\log^3n$ or $d/k<\log^3n$. It follows that certain random hypergraphs are covered too. Note that in Cases I and III we only use the first three properties of $(d,\alpha)$-expanding, and so, using Theorem \ref{thm:hyperlog-exp}, we obtain that $c(H^k(n,p))=O(\sqrt{n/k})$ in the following cases illustrated by Figure \ref{fig:resultsrandom}.

 \begin{figure}[H]
	\begin{tikzpicture}
		
		\begin{axis}[xmin=0,ymin=0,xmax=110,ymax=70,axis lines=middle, standard, axis line style={->},
			xlabel=$d/k$,
			ylabel=$k$,
			minor xtick={0.1,50},
			tick style={line width=1pt},
			xtick={0.1,50},
			xticklabels={$\vspace{1mm} \log n$, $log^3 n$},
			%minor ytick={0,5,10,1,25.2},
			ytick={0.1,20,47},
			yticklabels={3,$\log n$, $\log^3 n$}]
			
			\draw (25,33) node {Cases I and III work};
			\draw (50,10) node {No results.};
			\draw (75,33) node {Cases I, II and III work};
			\draw (25,56) node {Cases I, III and IV work};
			\draw (75,56) node {All four cases work};
			
			\addplot[name path=GGSS,very thick] coordinates {(0,20) (100,20)};
			\addplot[name path=GGS,very thick] coordinates {(0,47) (100,47)};
			\addplot[name path=NRS,very thick] coordinates {(50,20) (50,680)};
			
		\end{axis}
		
	\end{tikzpicture}
	\caption{In this figure we show a graph representing the ranges of $k$ and $d/k$ for which we've shown Conjecture \ref{conj:conj} to hold for the random hypergraph case.}
	\label{fig:resultsrandom}	
\end{figure}
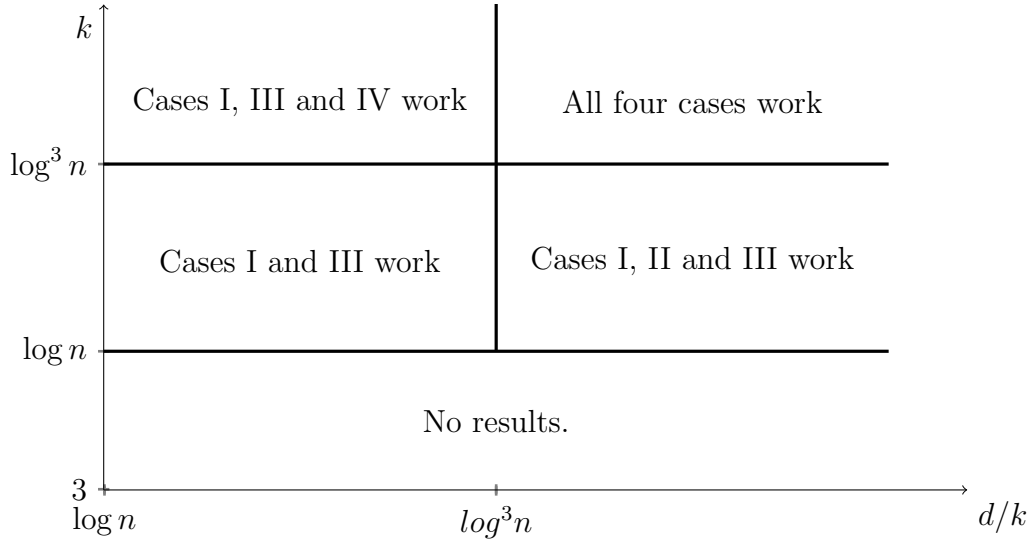

It remains open to show that Cases II and IV holds for any $d/k\geq \log n$ and $k\geq \log n$, and that all cases hold for $k\in[3,\log n]$.

%We believe that one may be able to prove the random case for $k<\log^3n$ and $d/k<\log ^3n$ using ideas from \cite{PW}. For $k\geq\log^3n$ and $d/k<\log ^3n$ one only needs to solve Case II, as the other ones already holds. In the same way, for $\omega(\log n)= k<\log^3n$ and $d/k\geq\log ^3n$, one only needs to solve Case IV. Although these cases have stronger assumptions which seems to make a proof easier, the expansion of edges and of vertices are so different that trying to have a fine control of the general expansion rate becomes rather difficult. We believe that these two cases will require novel ideas to be proven.

%		\nocite{*}
\bibliographystyle{plain}
\bibliography{main}

\end{document}